\documentclass[11pt]{amsart}

\usepackage{amsmath, amsthm, amscd, amssymb,amsfonts}
 \usepackage{color}
 \usepackage{mathrsfs}

\newtheorem{thm}{Theorem}[section]
\newtheorem*{thm*}{Theorem}
\newtheorem{cor}[thm]{Corollary}
\newtheorem*{cor*}{Corollary}
\newtheorem{lem}[thm]{Lemma}
\newtheorem{prop}[thm]{Proposition}

\newtheorem*{con*}{Conjecture}
\newtheorem*{prob*}{Problem}

\theoremstyle{definition}
\newtheorem{defn}[thm]{Definition}

\theoremstyle{remark}
\newtheorem{rem}[thm]{Remark}

\renewcommand{\L}{\Lambda}

\newcommand{\bbZ}{\mathbb{Z}}
\newcommand{\bbC}{\mathbb{C}}

\newcommand{\bbR}{\mathbb{R}}

\newcommand{\bbN}{\mathbb{N}}

\begin{document}
\title{Topological Rigidity of Quasitoric Manifolds}

\author{Vassilis Metaftsis}
\address{Department of Mathematics
University of the Aegean,
Karlovassi, Samos, 83200 Greece}
\email{vmet@aegean.gr}

\author[Stratos Prassidis]{Stratos Prassidis$^*$}

\email{prasside@aegean.gr}
\thanks{$^*$The research of the second author has been co-financed by the European Union (European Social Fund - ESF) and Greek national funds through the Operational Program "Education and Lifelong Learning" of the National Strategic Reference Framework (NSRF) - Research Funding Program: THALIS}

\maketitle

\begin{abstract}
Quasitoric manifolds are manifolds that admit an action of the torus that is locally as the standard action of $T^n$ on ${\bbC}^n$. It is known that the 
quotients of such actions are nice manifolds with corners. We prove that a class of locally standard
manifolds, that contains the quasitoric manifolds, 
is equivariantly rigid i.e., that any other manifold that is $T^n$-homotopy equivalent to a quasitoric manifold, is $T^n$-homeomorphic to it.
\end{abstract}

\section{Introduction}

Toric varieties are studied extensively in algebraic geometry and combinatorics (\cite{fu}, \cite{oda}). The main tool in their study is the simplicial complex that is 
determined by the fan of the toric variety. This simplicial complex is actually the quotient of the toric variety by the torus action.
The combinatorial properties of the simplicial complex reflect the algebraic and geometric properties of the variety and
vice versa. A topological analogue of toric varieties was introduced by Davis--Januszkiewicz (\cite{da-ja}) called toric manifolds in the paper. For avoiding confusion with the terminology, later the term  {\sl quasitoric manifolds} became prominent for these spaces. The term ``toric manifold'' is reserved for the 
non-singular toric varieties. Quasitoric manifolds are manifolds 
that admit an action of the torus $T^n$ which is {\it locally standard} such that the quotient space is a simple polytope. Locally standard actions are those where,
locally, $T^n$ acts by the standard coordinate-wise multiplication on ${\bbC}^n$. As in the toric variety case, the combinatorial properties of the polytope 
provide information about the topological structure of the manifold. Furthermore, in certain cases, the manifolds can be reconstructed from the polytope and an appropriate assignment of subgroups of $T^n$ to the faces of the polytope. 

 In this paper, we consider a further generalization studying locally standard $T^n$-actions on manifolds.
 In this case, the quotient space is a nice manifold with corners. As before, we show that the combinatorial properties of the manifold with corners are
 reflected to the topology of the manifold. Also, the  manifold can be reconstructed by an appropriate assignment of subgroups of $T^n$ to its 
 faces. The main theorem of the paper is the following.
 
 \begin{thm*}[Main Theorem]
 Let $M^{2n}$ be a closed $T^n$-quasitoric manifold.
 Let $N^{2n}$ a closed locally linear $T^n$-manifold and $f: N^{2n} \to M^{2n}$  an equivariant homotopy equivalence. Then $f$ is
 equivariantly homotopic to an equivariant homeomorphism.
 \end{thm*}

Actually, the theorem is proved for a slightly more general class of locally standard torus manifolds (Theorem \ref{thm-main}).

The idea of the proof is the same an the one used in the Coxeter group case (\cite{mo-pr}, \cite{pr-sp}, \cite{ro}). 
After all, the reconstruction of the quasitoric and $T^n$-locally standard manifolds,
from their quotient spaces, is similar to the construction of the Coxeter complex of a Coxeter group, a similarity that was made precise in \cite{da-ja}. First it is 
proved that $N^{2n}$ is a $T^n$-locally standard manifold. Let $X = M^{2n}/T^n$ and $Y = N^{2n}/T^n$. Then $X$ and $Y$ are nice manifolds with corners and $f$ induces a map ${\phi}: Y \to X$ that
is a face-preserving homotopy equivalence. As in the references for the Coxeter group case, we show inductively that there is a face-preserving homotopy from
$\phi$ to a face-preserving homeomorphism $h$. The homeomorphism $h$ lifts to a $T^n$-homeomorphism between $N^{2n}$ and $M^{2n}$ that is homotopic to $f$. 

The main theorem, loosely, can be considered as a version of an equivariant or stratified Borel Conjecture. Let ${\pi}: M \to X$ be the quotient map.
Over the interior $\overset{\circ}{\sigma}$ of faces of $X$, the map $\pi$ is a fiber bundle with fiber $T_{\sigma}$, where $T_{\sigma}$ is the isotropy group of 
$\sigma$. So, $M^{2n}$ admits a stratification by open aspherical manifolds. 

There are rigidity results known for non-singular toric varieties (\cite{ma}, \cite{ma-suh}), for quasitoric manifolds
(\cite{wi1}, \cite{wi2}) and for $T^n$-locally standard manifolds (\cite{yo1}). In all the above the classification
is given using cohomological and combinatorial data associated to the spaces.

In \cite{yo1}, a generalization of the locally standard actions is given, called local torus actions. Our methods do not directly generalize to this case.
In \cite{yo}, the generalization of the quotient map ${\pi}: M^{2n} \to X$ is given.  It is called local standard torus fibration. Again, our methods can not be
applied directly to the stratified rigidity problem for  such $M^{2n}$.

\section{Preliminaries and Notation}\label{sec-not}

We consider $S^1$ as the standard subgroup of ${\bbC}^*$, the multiplicative group of non-zero complex numbers. Furthermore $T^n < ({\bbC}^*)^n$. We refer to
the standard representation of $T^n$ by diagonal matrices in $U(n)$ as the standard action of $T^n$ on ${\bbC}^n$. 
The orbit space of the action is the positive cone:
$${\bbR}_+^n = \{(x_1, x_2, \dots, x_n):\; x_i \ge 0\}.$$

\begin{defn}
Let $M^{2n}$ be a $2n$-dimensional manifold with an action of $T^n$. The action
is called {\it locally standard} if for every $x\in M^{2n}$ there is a $T^n$ invariant
neighbourhood $U$ of $x$ and a homeomorphism $f: U\rightarrow W$ where
$W$ is an open set in ${\bbC}^n$ invariant under the standard action of $T^n$, 
and an automorphism $\phi: T^n\rightarrow T^n$ such that $f(ty)=\phi(t)f(y)$ for
all $y\in U$.  We call $M^{2n}$ a $T^n$-locally standard manifold. We will consider
only closed $T^n$-locally standard manifolds.

\end{defn}

\begin{defn}
An action $G\times M\rightarrow M$ is effective if there is no non-trivial element of $G$
that stabilizes $M$ pointwise. In other words the intersection of all isotropy subgroups
is trivial.
\end{defn}

\begin{rem}
\begin{enumerate}
\item If the action of $T^n$ on an even dimensional manifold $M^{2n}$ is effective and it does not have any finite isotropy groups, then the action is locally standard by the slice theorem (\cite{yo}, Example 2.1).
\item If  $M^{2n}$ is smooth and $H^{\text{odd}}(M) = 0$, then the action is locally standard (\cite{ma-pa}).
\end{enumerate}
\end{rem}

The next definition formalizes the local properties of the quotient space of a locally standard $T^n$-
action. The following definition is in \cite{davis}.

\begin{defn}
A space $X$ is an $n$-manifold with corners if it is a Hausdorff, second countable space equipped with an atlas of open sets each one 
homeomorphic to an open subset of ${\bbR}_+^n$ such that the overlap maps are local homeomorphisms that preserve the natural stratification of ${\bbR}^n_+$.
\end{defn}

The quotient of a locally standard action is a {\it manifold with corners} (\cite{ma-pa}). 

\begin{rem}
For any $n$-manifold with corners $X$ we have the following.
\begin{enumerate}
\item For each $x\in X$ and a chart $\sigma$, define $c(x)$ to be the number of coordinates of ${\sigma}(x)$ that are $0$. The number $c(x)$ is independent of the 
choice of the chart and so $c$ defines a map $c:X\rightarrow {\bbN}$. For $0 \le k \le n$, a connected component of $c^{-1}(k)$ is a stratum of codimension $k$. The closure of a stratum is called a closed stratum. 
\item Let $x\in X$. Define
$$Y(x) = \{ C:\; C \; \text{closed codimension-one stratum that contains} \; x\}.$$
The manifold with corners $X$ is called {\it nice} if $|Y(x)| = 2$, whenever $c(x) = 2$.
\item The slice theorem implies that the quotient space of a locally standard $T^n$-action is a nice manifold with corners (\cite{ma-pa}).
\item A {\it facet} in an $n$-manifold with corners is the closure of a connected component of the codimension $1$ stratum.
 A non-empty intersection of $k$ facets is called a codimension-$k$ {\it preface} 
($k = 1, \ldots, n$). In general, prefaces of codimension $> 1$ may be disconnected. A connected component of a preface is called a {\it face}. If $G$ is a subface of $F$, we write $G < F$ or
$F > G$.
The manifold $X$ itself is
considered to be a codimension-$0$ face. The $k$-{\it skeleton} of a manifold with corners $X$ is the set of all faces of codimension greater than or equal to $k$ and it is denoted
by $X^{(k)}$.

\end{enumerate}
\end{rem}

The following remark summarizes the connection between $T^n$-locally standard manifolds and
manifolds with corners.

\begin{rem}
\begin{enumerate}
\item Let $M^{2n}$ be  a closed $T^n$-locally standard manifold. Then the quotient $X^n = M^{2n}/T^n$ is 
a compact nice manifold with corners (\cite{ma-pa}, \cite{yo1}).

\item As we mentioned already, quasitoric manifolds are $T^n$-locally standard manifolds with the property that the quotient space is not just a manifold with corners but it is a simple polytope.

\item Let $M^{2n}$ be a $T^n$-locally standard manifold and ${\pi}: M^{2n} \to X^n$ the orbit map.

Then points in $M^{2n}$, with the same isotropy groups, are mapped to the relative interior of a preface of $X^n$.
Thus the action of $T^n$ is free over the open stratum of $X^n$ and the vertices of $X^n$ i.e. the $0$-dimensional faces, correspond to the fixed points of the action.

\end{enumerate}
\end{rem}

\begin{defn}
Let $M^{2n}$ be a $T^n$-locally standard manifold, $X = M/T^n$ the quotient 
manifold with corners and ${\pi}: M^{2n} \to X^n$ the quotient map. Then $M^{2n}$ is called
a $T^n$-manifold over $X^n$.
\end{defn}

Let ${\pi}: M^{2n} \to X^n$ be the projection defined above. A codimension-$1$ connected component of a fixed point set of a circle in $T^n$ is called a {\it characteristic submanifold} of $M$. The images of the characteristic submanifolds are the facets of $X$.

\section{The canonical model}

We will show how to reconstruct the $T^n$-locally standard manifolds from a manifold with corners $X$ and some linear data on the set of facets of $X$. We use the construction in \cite{ma-pa} that generalizes
the construction of quasitoric manifolds in \cite{bp} and \cite{da-ja}. We write $T = T^n$.

First, we will see some of the properties of characteristic submanifolds of a $T$-locally standard manifold. We assume that $M$ is a closed $T$-locally linear manifold and thus, its quotient, $X$ is a compact manifold with corners.
Let $M_i^{2(n-1)} = {\pi}^{-1}(X_i)$ be the characteristic submanifolds (\cite{bp}, page 34), where
$X_i$ are the facets of $X$ ($i = 1, \ldots, k$). Let
$${\Lambda}: \{ X_1, \dots , X_k\} \to \{T':\; T' < T, 1\text{-dimensional}\}.$$
be defined as ${\Lambda}(X_i)$ to be the isotropy group of $M_i$.
More precisely, ${\Lambda}(X_i)$ has the form
$$T_{X_i} = \{ (e^{2{\pi}i{\lambda}_{1j}{\phi}}, \dots e^{2{\pi}i{\lambda}_{nj}{\phi}})\in T^n;\;
\phi \in \bbR\},$$
for some primitive vector $({\lambda}_1, \dots , {\lambda}_n)$ of ${\bbZ}^n$.
The main property of these data is  that (\cite{bp}, p. 34): 

\vspace{12pt}
\noindent
{\bf Property (*)}:
if $X_{i_1}{\cap}\dots {\cap} X_{i_m} \not= \emptyset $ then 
${\Lambda}(X_{i_1}) {\times} \dots {\times}{\Lambda}(X_{i_m}) \to T$
is injective.

\vspace{12pt}

Let $F$ be an $m$-face of $X$. Then $F = X_{i_1}{\cap}\dots {\cap} X_{i_{n - k}}$, for some facets of $X$.
We write $T_F = T_{X_{i_1}}{\times} \dots T_{X_{i_{n - k}}}$, which is an $(n - k)$-torus.
That construction defines a map between lattices, extending the map $\Lambda$ above (\cite{bp}, page 34).
$${\Lambda}: \{ F: \; F < X\} \to \{T': \; T' < T\}, \quad F \mapsto T_F.$$

Now, we give the inverse of the construction (\cite{bp}, Construction 2.2.2). We start with a compact manifold with corners $X$ and a map $\Lambda$ that satisfies Property (*) above. Such a pair $(X, {\Lambda})$ is called 
a {\it characteristic pair} and $\Lambda$ the {\it characteristic map}.
For $x\in X$, we denote by $F(x)$ the smallest face of $X$ that contains $x$ in its relative interior.

Define:
$$M_X({\Lambda}) = T{\times}X/{\sim}, \;\; (t, x) \sim (t', x') \Longleftrightarrow x = x', \; \text{and}\; t^{-1}t'\in T_{F(x)}.$$
The space  $M_X({\Lambda})$ is a closed manifold and the torus $T$ acts on it  by acting on the first coordinate. 
Actually, the space $M_X({\Lambda})$ is a smooth locally standard manifold (\cite{bp}, Construction 2.2,
\cite{ma-pa}, Proposition 4.5). That means that $M_X({\Lambda})$ is locally standard and the global action of the torus is smooth.

The following result is implicit in \cite{bp}, page 34.

\begin{lem}\label{lem-fix}
Let $(X, {\Lambda})$ be a characteristic pair and $M_X({\Lambda})$ the canonical model
corresponding to it. Let ${\pi}: M_X({\Lambda}) \to X$ the quotient map. Then for a face $F$ of $X$ with corresponding 
group $T_F$, the fixed point set of $T_F$ is given by:
$$M_X({\Lambda})^{T_F} = {\pi}^{-1}(F) = \{[t, x]: \; t\in T, \; x\in F\} \subset M_X({\Lambda}).$$
\end{lem}

\begin{proof}
The proof is also in \cite{bp}. 
First we will show that ${\pi}^{-1}(F) \subset M_X({\Lambda})^{T_F}$. Let $[t, x]\in {\pi}^{-1}(F)$.
Then $x \in F$, which implies that $T_{F(x)} > T_F$. For $t' \in T_F$, $t'[t, x] = [t't, x]$. But
$t't.t^{-1} = t'\in T_F < T_{F(x)}$, which implies $[t't,x] = [t, x]$. 
Thus $[t, x] \in M_X({\Lambda})^{T_F}$.

Now we show the inverse inclusion $M_X({\Lambda})^{T_F} \subset {\pi}^{-1}(F)$. Let
$[t, x]$ be fixed by $T_F$. Then, for $t'\in T_F$,
$$t'[t, x] = [t't, x] = [t, x] \Rightarrow t't.t^{-1} = t' \in T_{F(x)} \Rightarrow T_F < T_{F(x)} \Rightarrow
F(x) < F$$
the last relation assets that $F(x)$ is a face of $F$. Thus $x \in F$, which completes the argument.
\end{proof}

The following results compares a $T$-locally standard manifold with its canonical model (\cite{yo1}, Section 5). In \cite{yo1}, Lemma 5.2 and Theorem 5.5, it is shown that the two manifolds $M$ and 
$M_X({\Lambda})$ are $T$-homeomorphism, with a homeomorphism covering the 
identity on $X$, if and only if a class $e(M, X) \in 
\check{\text{H}}^1(X, \mathscr{S}_{(X, {\Lambda})})$, called the
{\it Euler class}, vanishes. Here the cohomology theory is \v{C}ech cohomology with coefficients 
defined as follows. Let ${\pi}_{\Lambda}: M_X({\Lambda}) \to X$ be the quotient map.
The sheaf  $\mathscr{S}_{(X, {\Lambda})}$ assigns to every open set $U$ of $X$ so that 
${\pi}_{\Lambda}|_U: {\pi}_{\Lambda}^{-1}(U) \to U$ is trivial, 
the section $\text{Sec}({\pi}_{\Lambda}|_U)$ of ${\pi}_{\Lambda}|U$.

\begin{lem}[\cite{yo1}]\label{lem-canonical}
Let $M^{2n}$ be a $T$-locally standard manifold, ${\pi}: M \to X$ the orbit map and 
$M_X({\Lambda})$ the canonical model associated to the action. Then the following are equivalent:
\begin{enumerate}
\item There is a $T$-homeomorphism $h: M \to M_X({\Lambda})$
 covering the identity on $X$,
\item The orbit map ${\pi}: M \to X$ admits a section.
\item The Euler class $e(X, M) \in \check{\text{H}}^1(X, \mathscr{S}_{(X, {\Lambda})})$ vanishes.
\end{enumerate}
If any of the conditions above hold we say that the pair $(M, X)$ splits.
\end{lem}

\begin{rem}\label{rem-homeomorphism}
\begin{enumerate}
\item In \cite{da-ja} the existence of the $T$-homeomorphism between $M$ and $M_X({\Lambda})$ was proved for quasitoric manifolds. Thus when $X$ is a simple polytope, $\check{\text{H}}^1(X, \mathscr{S}_{(X, {\Lambda})}) = 0.$ So for every quasitoric manifold $M$, $(M, X)$ splits.
\item In \cite{ma-pa} it was shown that the $T$-homeomorphism exists under
the condition that $M$ is a smooth $T^n$- locally standard manifold and $H^2(X, {\bbZ}) = 0$. So 
Thus, in this case, the pair $(M, X)$ splits.
\item In \cite{yo1} the result was stated for manifolds that admit a local torus action.
\end{enumerate}
\end{rem}

Now we investigate the natural properties of the construction. 

\begin{defn}
Let ${\phi}: Y \to X$ a map between manifolds with corners. 
\begin{enumerate}
\item $\phi$ is called {\it skeletal} if it preserves skeleta i.e.
${\phi}(Y^{(k)}) \subset X^{(k)}$. 
\item $\phi$ is called {\it face preserving} if, for each face $F$ to $X$, ${\phi}(F)$ is a face of $Y$.
\end{enumerate}
\end{defn}

\begin{rem}
\begin{enumerate}
\item Similarly, a homotopy ${\phi}_t: Y \to X$,  is called skeletal (face preserving) if the map at each level is skeletal (face preserving). 
\item Notice that  face-preserving maps or homotopies are skeletal. 
\end{enumerate}
\end{rem}

\begin{prop}\label{rem-face-preserving}
If ${\phi}: Y \to X$ is a skeletal homotopy equivalence then it is necessarily face-preserving 
homotopy equivalence. 
\end{prop}

\begin{proof}
We will show that $\phi$ is face preserving. That will imply that each level of the skeletal homotopies will be also face preserving. That is because there is a skeletal homotopy
from each level to the corresponding identity map.

We use induction on the dimension of the faces. The inductive statement is:

\noindent
$\phi$ induces a bijection between the sets of $k$-faces of $Y$ and $X$.

The statement is obviously true for the $0$-faces, 
which are points. We assume that the statement is true for $Y^{(k-1)}$. That means that
$\phi$ induces a bijection between the $(k -1)$-faces of $Y$ and $X$.
Let $F'$ be a $k$-face of $Y$. Then ${\phi}(F') \subset X^{(k)}$, because $\phi$
is skeletal. First notice that 
${\phi}(F') \not\subset X^{(k-1)}$ because, in that case, at least two $(k -1)$-subfaces of $F'$ will
map to the same $(k -1)$-face of $X$. Also,  by the induction hypothesis, 
${\phi}({\partial}F')$ will be mapped to a boundary of a $k$-face in $X$, completing the proof.
\end{proof}

\begin{prop}\label{prop-natural}
 Let $(X^n, {\Lambda})$ and $(Y^n, {\Lambda}')$ be two characteristic pairs and ${\sigma}: T \to T$ a continuous automorphism. Let ${\phi}: Y \to X$ be
 a face-preserving map that satisfies ${\sigma}(T_{Y_i}) < T_{{\phi}(Y_i)}$ for each facet $Y_i$ of $Y$. 
 Then ${\phi}$ induces a $\sigma$-equivariant map 
 ${\phi}_*: M_{Y}({\L}') \to M_{X}({\L})$.
\end{prop}

\begin{proof}
Define that map ${\phi}_*$, the obvious way:
$${\phi}_*: M_{Y}({\L}') \to M_{X}({\L}), \; {\phi}_*([t,y]) = [{\sigma}(t), {\phi}(y)].$$
We need to show that the map is well-defined. Let $[t, y] = [t', y]$ in $M_Y({\L}')$. Then $t^{-1}t' \in 
T_{F'(y)}$. Also,  $F'(y) = Y_{i_1}{\cap} \dots {\cap}Y_{i_m}$, as intersection of facets. Then 
$${\phi}(y)\in  {\phi}(Y_{i_1}) {\cap} \dots {\cap} {\phi}(Y_{i_m}) = X'.$$
Since the map $\phi$ is face-preserving, there are facets $X_i$, $i = 1, \dots s$, of $X$ such that:
$$F({\phi}(y)) = {\phi}(Y_{i_1}) {\cap} \dots {\cap} {\phi}(Y_{i_m}) {\cap} X_1 {\cap} \dots X_s.$$
Therefore $T_{X'} < T_{F({\phi}(y)}$.
Also, we have that ${\sigma}(T_{Y_i}) < T_{{\phi}(Y_i)}$, and thus 
$$T_{F({\phi}(y))} > T_{X'} > {\sigma}(T_{F'(y)})$$
Therefore ${\sigma}(t^{-1}t')\in T_{F({\phi}(y))}$ and $[{\sigma}(t), {\phi}(y)] = [{\sigma}(t'), {\phi}(y)]$ in $M_X({\L})$.

By the construction, the map is obviously $\sigma$-equivariant.
\end{proof}

\begin{rem}\label{rem-isomorphism}
We use the above notation.
Let $(X, {\L})$ and $(Y, {\L}')$ are two characteristic pairs and ${\phi}: X \to Y$ a face-preserving map
such that ${\sigma}(T_{Y_i}) < T_{{\phi}(Y_i)}$ for each facet $Y_i$ of $Y$. Then necessarily
${\sigma}(T_{Y_i}) = T_{{\phi}(Y_i)}$. That is because both ${\sigma}(T_{Y_i})$ and  $T_{{\phi}(Y_i)}$
 are maximal subgroups of $T$ isomorphic to $S^1$. Thus they must be equal.
\end{rem}

\begin{cor}\label{cor-homotopy}
If ${\phi}_s: Y \to X$, $s\in [0, 1]$, is a face-preserving homotopy so that ${\sigma}(T_{Y_i}) < T_{{\phi}_s(Y_i)}$, for each $s$ and each facet $Y_i$ of $Y$. Then ${\phi}_{0,*} \simeq_{\sigma} {\phi}_{1,*}$.
\end{cor}

We now investigate the reverse construction. We will do that in a much more restricted setting. First we need a proposition. 

\begin{prop}\label{prop-reverse}
Let $(X, {\Lambda})$ and $(Y, {\Lambda}')$ be two characteristic pairs such that all the faces of
$X$ and $Y$ are contractible manifolds with boundary. Let $f: M_Y({\L}') \to M_X({\L})$ a $\sigma$-equivariant homotopy equivalence, with $\sigma$ as before. Then
\begin{enumerate}
\item The map $\phi$ induced on the faces is face-preserving homotopy equivalence.
\item ${\sigma}(T_{F'}) = T_{{\phi}(F')}$ for each facet $F'$ of $Y$.
\item There is a $\sigma$-equivariant homotopy such that $f \simeq_{\sigma} {\phi}_*$.
\end{enumerate}
\end{prop}

\begin{proof}
The equivariance implies that $\phi$ is skeletal. 
Then Proposition \ref{rem-face-preserving} shows that the map is face preserving
homotopy equivalence.

For (2), 
let $Y_i$ be a facet of $Y$ and $T_{Y_i}$ its
isotropy group. Then, equivariance again, implies
that the isotropy group of ${\phi}(Y_i)$ contains ${\sigma}(T_{X_i})$. 
Together 
with Remark \ref{rem-isomorphism} it shows that ${\sigma}(T_{Y_i}) = T_{{\phi}(Y_1)}$. 
Let $F'$ be a face of $Y$, $F' = Y_{1}{\cap}\dots \cap Y_{m}$ written as an intersection
of facets. Since $\phi$ induces a bijection on  faces,
${\phi}(F') = {\phi}(Y_1) \cap \dots \cap {\phi}(Y_m)$, as an intersection of facets.
Then 
$${\sigma}(T_{F'}) = {\sigma}(T_{Y_1}{\times} \dots {\times}T_{Y_m}) =
{\sigma}(T_{Y_1}){\times} \dots {\times}{\sigma}(T_{Y_m}) = 
T_{{\phi}(Y_1)} {\times} \dots {\times} T_{{\phi}(Y_m)} = T_{{\phi}(F')}.$$

For (3), notice that the map $f$ induces a map 
$$f_Y: Y \to M_X({\L}), \ \mbox{with}\ \ f_Y(y) = f([1, y]) = [t_y, {\phi}(y)],$$
for some $t_y \in T$. 

For each face $F'$ of $Y$, we write $T_F = {\sigma}(T_{F'})$, where $F = {\phi}(F')$.
The map $f$ induces 
a homotopy equivalence 
$$f^{T_{F'}}: M_Y({\L}')^{T_{F'}} \to M_X({\L})^{{\sigma}(T_{F'})} = M_X({\L})^{T_F}$$
Write
$$f_{F'}: F' \xrightarrow{{\iota}_{F'}} M_Y({\L}')^{T_F'} \xrightarrow{f^{T_{F'}}}
M_X({\L})^{T_F}$$ 
where ${\iota}_{F'}(y) = [1, y]$. Explicitly, for $y \in F'$,
$$f_{F'}(y) = f^{T_{F'}}\circ{\iota}_{F'}(y) = f^{T_{F'}}([1, y]) 
= [t_y, {\phi}(y)].$$
Notice that ${\phi}(y) \in F$ (Lemma \ref{lem-fix}).
Also, define 
$${\phi}_{F'}: F'  \xrightarrow{{\iota}_{F'}} M_Y({\L}')^{T_F'} \xrightarrow{{\phi}_*^{T_{F'}}}
M_X({\L})^{T_F}, \;\;{\phi}_{F'}(y) = [1, {\phi}(y)]$$

Now let $y_0 \in F'$ be a base point. 
Let $c'_s$ be a contracting homotopy, starting from the identity on $F'$ and ending to the constant map at 
$y_0$. Similarly, choose a contracting homotopy $c_s$ from the identity on $F$ to the constant map to ${\phi}(y_0)$.
Let $W_F = T/T_F$ be the Weyl group of $T_F$. Choose
 a path $\beta$ in $W_F$ with ${\beta}(0) = t_{y_0}T_F$ and 
${\beta}(1) = T_F$. Then
define a homotopy ${\chi}_{F'}: F' {\times}I \to M_X({\Lambda})^{T_F}$, as follows
$${\chi}_{F'}(y, s)= \left\{
\begin{array}{ll}
f([1, c'_{2s}(y)]), & \displaystyle{0 \le s \le \frac{1}{2}} \\[2ex]
[\bar{\beta}(2s - 1), c_{2 - 2s}({\phi}(y))], & \displaystyle{\frac{1}{2} \le s \le 1}
\end{array}
\right.$$
with $\bar{\beta}(2s - 1)$ a coset representative of $\beta(2s - 1)$.
Notice that
\begin{enumerate}
\item ${\chi}_{F'}$ is well defined:
\begin{enumerate}
\item Let $\bar{\beta}_i(2s - 1) \in T$, $i = 1, 2$, be two elements of the coset ${\beta}(2s - 1).$
Then,  there is $t\in T_F$ such that $\bar{\beta}_1(2s - 1) =  t\bar{\beta}_2(2s - 1)$.
But 
$$c_{2 - 2s}({\phi}(y)) \in F \Rightarrow F(c_{2 - 2s}({\phi}(y))) \le F \Rightarrow
T_{F(c_{2 - 2s}({\phi}(y)))} \ge T_F \Rightarrow t \in T_{F(c_{2 - 2s}({\phi}(y)))}.$$
Therefore $\bar{\beta}_1(2s-1)(\bar{\beta}_2(2s-1))^{-1}=t\in T_{F(c_{2 - 2s}({\phi}(y)))}$ and so 
$$[\bar{\beta}_1(2s - 1), c_{2 - 2s}({\phi}(y))] = [\bar{\beta}_2(2s - 1), c_{2 - 2s}({\phi}(y))]$$ by definition. Hence  the homotopy does not depend on choice of the representative 
of ${\beta}(2s - 1)$ in $W_F$.
 \item For $s = 1/2$, the two branches of the function read:
\begin{enumerate}
\item $f([1, c'_1(y)]) = f([1, y_0]) = [t_{y_0}, {\phi}(y_0)]$.
\item $[\bar{\beta}(0), c_1({\phi}(y))] = [t_{y_0}, {\phi}(y_0)]$.
\end{enumerate}
\end{enumerate}
\item ${\chi}_{F'}(y, 0) = f([1, y]) = [t_y, {\phi}(y)] = f_{F'}(y)$.
\item ${\chi}_{F'}(y, 1) = [\bar{\beta}(1), {\phi}(y)] = [1, {\phi}(y)] = {\phi}_{F'}(y)$
\end{enumerate}

For each face $F'$, we will construct a homotopy 
$h_{F'}: F'{\times}[0,1] \to M_X({\Lambda})^{T_F}$ from $f_{F'}$ to ${\phi}_{F'}$ such
that:
\begin{enumerate}
\item $h_{F'}: {f}_{F'} \simeq \phi_{F'}$.
\item For $G'$ a subface of codimension $1$ of $F'$ (denoted $G'<F'$), the restriction of $h_{F'}$ to $G'$ has the form:
$$h_{F'}|G'(y, s) = \left\{
\begin{array}{ll}
h_{G'}(y, 2s), & \displaystyle{0 \le s \le \frac{1}{2}} \\[2ex]
[1, {\phi}(y)], & \displaystyle{\frac{1}{2} \le s \le 1}
\end{array}
\right.$$
\end{enumerate}
We use the notation $h_{F'}|G' = h_G'*{\phi}_*$ for the concatenation above.

The construction is done inductively. For a 0-cell $v'$, $\text{Im}(f_{v'}) = \{[1, {\phi}(v')]\}$. So the homotopy on the 0-skeleton
is the stationary homotopy. Let $F'$ be an 1-cell. Then ${\chi}_{F'}$ induces a homotopy from
$f_{F'}$ to ${\phi}_{F'}$. 
 Then there is a homotopy $h_{F'}$ 
 $$h_{F'}: f_{F'} \simeq {\phi}_{F'}, \; \text{rel}({\partial}F').$$

Now assume that we defined the homotopy ${\partial}h_{F'}$ over the boundary
of a $k$-face $F'$, $k > 1$. The second property of the homotopies
$h_{G'}$, for $G' < F'$, allows the assembly of the homotopies $h_{G'}$ to construct 
a homotopy $h_{{\partial}F'}$ on ${\partial}F'$.  The homotopy $h_{{\partial}F'}$ has the property
that, for each $G' < F'$ of codimension $1$, $h_{{\partial}F'}|G' = h_{G'}$.
Notice that, for each $G' < F'$, we have $G = {\phi}(G') < F$, $T_G > T_F$ and
$M_X({\Lambda})^{T_G} <  M_X({\Lambda})^{T_F}$. 
Thus we have a homotopy, for each $G' < F'$ of codimension $1$,
$$G'{\times}I \xrightarrow{h_{G'}} M_X({\Lambda})^{T_G} 
\xrightarrow{j_G} M_X({\Lambda})^{T_F}$$
That means that the homotopy $h_{{\partial}F'}$ induces a homotopy (also denoted 
$h_{{\partial}F'}$)
$$h_{{\partial}F'}: {\partial}F' {\times} I \to M_X({\Lambda})^{T_F}.$$
Using the homotopy extension property  
we have a homotopy
$g_{F'}: F'{\times}[0, 1] \to M_X({\Lambda})^{T_F}$ such that
\begin{enumerate}
\item $g_{F'}(y, 0) = f_{F'}(y)$.
\item For each $G ' < F'$, a face of $F'$ of codimension $1$, $g_{F'}|{G'} = h_{G'}$.
\end{enumerate}
Set $g_{F',1} = g_{F'}(-, 1)$. Now $g_F: f_{F'} \simeq g_{F',1}$ and ${\chi}_{F'}: f_{F'} \simeq 
{\phi}_{F'}$. Thus there is a homotopy $g_{F',1} \simeq {\phi}_{F'}$. Also, for $y \in {\partial}F'$, $y$
belongs to a codimension $1$ subface of $F'$ (it does not matter which one) and
$$g_{F',1}(y) = g_{F'}(y, 1) = h_{{\partial}F'}(y, 1) = [1, {\phi}(y)] = {\phi}_{F'}(y).$$
Thus, there is a homotopy ${\psi}_{F'}: g_{F',1} \simeq {\phi}_{F'}$, rel${\partial}F'$. Define the homotopy $h_{F'} = g_{F'}*{\psi}_{F'}$, the concatenation of the two homotopies. Then
\begin{enumerate}
\item $h_{F'}(y, 0) = f_{F'}(y)$.
\item $h_{F'}(y, 1) = {\phi}_{F'}(y)$.
\item If $y \in {\partial}F$, then
\begin{enumerate}
\item For $0 \le s \le 1/2$, $h_{F'}(y, s) = h_{{\partial}F'}(y, 2s)$.
\item For $1/2 \le s \le 1$, $h_{F'}(y, s) = {\psi}_{F'}(y, 2s) = {\phi}_{F'}(y)=[1,\phi(y)]$.
\end{enumerate}
\end{enumerate}
Thus, $h_{F'}$ satisfies all the conditions required. Working inductively we get that there is a homotopy $h: Y{\times}I \to M_X({\Lambda})$ such that
\begin{enumerate}
\item $h(y, 0) = f_Y(y) = f([1, y]) = [t_y, {\phi}(y)]$.
\item $h(y, 1) = {\phi}_*([1, y]) = [1, {\phi}(y)]$
\item For each face $F'$ of $Y$, $\text{Im}(h|F') \subset M_X({\Lambda})^{T_F}$. 
\end{enumerate}
Define $H: M_Y({\Lambda}') {\times} I \to M_X({\Lambda})$, $H([t, y],s) = {\sigma}(t)h(y, s)$, which 
is the required homotopy $H: f \simeq {\phi}_*$.
\end{proof}

\section{Rigidity}

Set $T = T^n$. Let $M^{2n}$ be a locally standard closed manifold and $X = M/ T$ the 
corresponding nice manifold with corners. 
In this section, we assume that:
\begin{enumerate}
\item All the faces of $X^n$ (and $X^n$ itself) are contractible spaces. 
\item $\check{\text{H}}^1(X, \mathscr{S}_{(X, {\Lambda})}) = 0$.
\end{enumerate} 
That is the situation when $M$ is a quasitoric manifold (Remark \ref{rem-homeomorphism}). 
In this case (Lemma \ref{lem-canonical})
$M \cong_T M_X({\Lambda})$ for the characteristic map $\Lambda$ induced by the action
and the pair 
$(M, X)$ splits. Notice that 
the $T$-action of $T$ on $M_X({\Lambda})$ is effective and 
its isotropy groups are subtori of $T$.
Thus the same is true for $M$.
Let $N^{2n}$ be a closed $2n$-dimensional locally linear
$T$-manifold and $f: N^{2n} \to M^{2n}$
a $T$-equivariant homotopy equivalence with $G$-homotopy inverse $g$.

\begin{lem}\label{lem-dimension}
The action of $T$ on $N^{2n}$ is effective.
\end{lem}

\begin{proof}
We assume that that is not the case. So there is $t\in T$ that fixes $N^{2n}$ pointwise. Let $G ={\langle}t{\rangle}$. Then $N^G = N^{2n} \simeq M^G$ since $f$ is an equivariant homotopy equivalence. But $M^G$ is a closed proper submanifold of 
$M^{2n}$, because the action on $M^{2n}$ is effective. Thus ${\dim}(N^G) = {\dim}(M^G) < {\dim}(M^{2n}) = {\dim}(N^{2n})$, a contradiction.
\end{proof}

\begin{lem}\label{lem-isotropy}
The non-trivial isotropy subgroups of $N$ are subtori of $T$.
\end{lem}

\begin{proof}
Let $y\in N$ with isotropy group $T_y$

$$T_y = G_{i_1}{\times} \dots {\times} G_{i_k}$$
where $G_{i_j}$ is either a subtorus or a finite cyclic group. Set $T'$ be the maximal
subtorus contained in $T_y$. Assume that $T_y\gneqq T'$.

Since the isotropy groups of $M$ are subtori, $M^{T'} = M^{T_y}$ but $N^{T'} \supsetneq N^{T_y}$. 
But $M^{T'} \simeq N^{T'}$ and $M^{T_y} \simeq N^{T_y}$.
Since fixed point
sets are closed submanifolds without boundary, we have that
$$ \dim M^{T_y} = \dim N^{T_y } < \dim N^{T'} = \dim M^{T'}$$
Contradiction, because $ \dim M^{T_y} = \dim M^{T'}$.
\end{proof}

Using \cite{yo}, Example 2.1 and Lemmata \ref{lem-dimension} and \ref{lem-isotropy}, we see that $N$ is 
locally standard. For completeness, in our case, we will give an explicit proof.

\begin{prop}\label{prop-locally-standard}
The action of $T$ on $N^{2n}$ is locally standard. Thus, $N^{2n}/T = Y$ is a manifold with corners.
\end{prop}

\begin{proof}
Let $y \in N$ with isotropy group $T_y$. Then, Lemma \ref{lem-isotropy} implies that $T_y$ is either 
trivial or a subtorus of $T$.
Since the action is locally linear, there is a linear slice
$$s: T{\times}_{T_y} V \to TV$$
so that $V$ is a linear representation of $T_y$ and $TV$ is an open subset of $N$. Lemma \ref{lem-dimension}
implies that the action is effective. Thus the action on $TV$ is also effective (\cite{ggk}, Corollary B.42). So
the induced representation $T{\times}_{T_y} V$ of $T$ is faithful. We consider the last representation as
$${\phi}: T \to O(2n) \hookrightarrow U(n).$$
Let $d: T \to U(n)$ be the standard diagonal embedding.
Since $T$ is abelian, the representation $\phi$ is the direct sum of 1-dimensional representations. Thus
the representation can be realized $d{\circ}{\phi}$ where $\phi$ is a smooth self-monomorphism $\phi$ of $T$. 
Such a monomorphism is necessarily an isomorphism. For $y\in Y$, $TV$ is an open $T$-invariant neighborhood of $y$, $s^{-1}$ a homeomorphism to ${\bbC}^n$ and 
$$s^{-1}(ty) = d{\circ}{\phi}(t)s^{-1}(y).$$ 
Since $d$ is the standard action on ${\bbC}^n$, that implies that the action is locally standard.
\end{proof}

We denote by ${\Lambda}'$ the characteristic function defined by the $T$-action on $N^{2n}$. Also, by Proposition 
\ref{rem-face-preserving}, the map $f$ induces a a face-preserving homotopy equivalence ${\phi}: Y \to X$.
By Proposition \ref{prop-reverse}, the map $f$ is $T$-homotopic to ${\phi}_*$.

We need a version of the Poincar{\'e} Conjecture. For an $n$-dimensional manifold
with boundary $(M,\partial M)$ the relative structure set ${\mathcal{S}}(M,\partial M)$
is the set of equivalence classes of pairs $(N,f)$ with $N$ an $n$-dimensional manifold with boundary and $f:N\rightarrow M$ a homotopy equivalence such that $\partial f:
\partial N\rightarrow\partial M$ is a homeomorphism. The equivalence relation is generated by homeomorphisms.

For the following lemma, the structure set is defined as follows:
$$\mathcal{S}(M, {\partial}M) = \{f:(X^n, {\partial}X) \to (M, {\partial}M)\mid f \;\text{a homotopy
equivalence}, \; f|_{{\partial}X} \;\text{homeomorphism}\}/\sim$$
where the equivalence relation is given by homeomorphisms.

\begin{lem}\label{lem-poincare}
Let $(M, {\partial}M)$ be a compact contractible $n$-manifold with boundary. Then the relative structure set $\mathcal{S}(M, {\partial}M) = *$, $n \not= 3$. If $n = 3$ and $M \subset S^3$, then $(M, {\partial}M) \cong (D^3, S^3)$.
\end{lem}

\begin{proof}
For $n = 1, 2$, the result is obvious. For $n \ge 4$,
there is the surgery exact sequence:
$$\dots \to [(M{\times}I, {\partial}M{\times}I), (G/Top, *)] \to L_{n+1}({\bbZ}) \to \mathcal{S}(M, {\partial}M) \to [(M, {\partial}M), (G/Top, *)] \to L_n({\bbZ})$$
(for $n \ge 5$ this is the classical surgery exact sequence (\cite{wa}), for $n = 4$ the result follows from the results 
of Freedmann (\cite{fr}), see also Kirby-Taylor (\cite{kt}, {\S}7).
But $M/{\partial}M \cong S^n$ and $\mathcal{S}(S^n) = *$ for $n \ge 4$. We also have commutative diagrams
$$ \begin{CD}
[(M{\times}I, {\partial}M{\times}I), (G/Top, *)] @>>> L_{n+1}({\bbZ}) \\
@V{\cong}VV @| \\
 [(M/{\partial}M{\times}I, {\partial}M{\times}I), G/Top] @>>> L_{n+1}({\bbZ}) 
\end{CD}
\qquad
\begin{CD}
[(M, {\partial}M), (G/Top, *)] @>>> L_n({\bbZ}) \\
@V{\cong}VV @| \\
[M/{\partial}M, G] @>>> L_n({\bbZ})
\end{CD}
$$
So the vanishing of the structure set of the sphere implies that, in the exact sequence, the first map is
onto and the last map is into. Thus $\mathcal{S}(M, {\partial}M) = *$.

For $n = 3$, the results of Perelman (\cite{pe1}, \cite{pe2}, \cite{pe3}) imply that there are face disks
and spheres in dimension 3. Thus $(M, {\partial}M) \cong (D^3, S^2)$ (\cite{pr-sp}, Lemma 5.2 and \cite{ro}, Proof of Theorem 3.10)
\end{proof}

\begin{lem}\label{lem-homeomorphism}
The map ${\phi}: Y \to X$ is face-preserving homotopic to a face-preserving homeomorphism.
\end{lem}

\begin{proof}
We will use the method that was used in \cite{mo-pr}, \cite{pr-sp}, and \cite{ro} to show that $\phi$ is face-preserving homotopic to a face-preserving homeomorphism.  
We will construct a face-preserving homotopy by induction on faces. Notice that because $f$ is a $T$ equivariant homotopy, the map is face-preserving (Proposition \ref{rem-face-preserving}). That means that if $F_1$ is a face of $Y$ of codimension-$k$, then $\phi$ maps $F_1$ to $F_2$ 
where $F_2$ is a face of $X$ of codimension-$k$. Also, each closed face is homeomorphic to a contractible manifold with boundary.

We start the induction.
The zero faces correspond to the $T$-fixed point sets. Thus, we have the same number of zero faces. The restriction of $\phi$ to zero faces is a homeomorphism. Now, let a face $F_1$ be a face of $Y$ and ${\partial}F_1$ its boundary. We assume that there is face-preserving homeomorphism
$h_{{\partial}F_1}$ face-preserving homotopic to ${\phi}|_{{\partial}F_1}$. Using the homotopy extension property, there is a map ${\phi}': F_1 \to F_2$ that is homotopic to ${\phi}|_{F_1}$ and
it extends the map $h_{{\partial}F_1}$. Because all the maps and homotopies are face-presrving at the boundary, they are face-preserving in the closed face $F_1$. By Lemma
\ref{lem-poincare}, ${\phi}'$ is homotopic to a homeomorphism relative to the boundary. As before, all homotopies are face-preserving. Continuing this way, we get a face-preserving
homeomorphism $h: Y \to X$ that is face-preserving homotopic to $\phi$. 
\end{proof}

Lifting the maps and the homotopies to the canonical models, we have the following.

\begin{cor}\label{cor-homeomorphism}
With the above notation, the map ${\phi}_*: N_Y({\Lambda}') \to M_X({\Lambda})$ is $T$-homotopic to a $T$-homeomorphism.
\end{cor}

\begin{lem}\label{lem-cech}
Let $\phi$ and $h$ be the maps of Lemma \ref{lem-homeomorphism}. Then
$\mbox{\v H}^1(Y, \mathscr{S}_{(Y, {\Lambda}')}) = 0$.
\end{lem}

\begin{proof}
Remark \ref{rem-isomorphism} implies that the map $h$ induces an isomorphism of characteristic pairs $(Y, {\Lambda}')$ and
$(X, {\Lambda})$ in the sense of \cite{yo1}, Section 4. By the remarks before Lemma 5.9 in \cite{yo1}, the map $h$ induces an
isomorphism:
$$h^*: \mbox{\v H}^1(X, \mathscr{S}_{(X, {\Lambda})})  \to\check{\text{H}}^1(Y, \mathscr{S}_{(Y, {\Lambda}')}).$$
The assumption is that the left hand side vanishes. The lemma follows.
\end{proof}

Lemma \ref{lem-canonical} implies the following.

\begin{cor}\label{cor-canonical}
The pair $(N, Y)$ splits i.e., there is an $T$-homeomorphism $N \cong_{T} N_Y({\Lambda}')$.
\end{cor}

\begin{thm}\label{thm-main}[Rigidity for $T^n$-Locally Standard Manifolds]
Let $M$ be a closed $T^n$-locally linear manifold over a manifold with corners $X$ and characteristic 
map $\Lambda$. We assume that
\begin{enumerate}
\item  All the faces of $X$ (and $X$ itself) are contractible manifolds with corners. 
\item $\mbox{\v H}^1(X, \mathscr{S}_{(X, {\Lambda})}) = 0$.
\end{enumerate}
Let  $N$ a locally linear closed $T^n$-manifold and $f:N \to M$ a $T^n$-equivariant homotopy equivalence. Then $f$ is $T^n$-homotopic to a $T^n$-homeomorphism.
\end{thm}

\begin{proof}
Lemma \ref{lem-canonical} implies that the pair $(M, X)$ splits.
From Proposition \ref{prop-locally-standard}, the action of $T^n$ on $N$ is locally standard. Corollary \ref{cor-canonical} implies that $(N, Y)$ splits.

 Then the map $f$ induces a face-preserving map ${\phi}: Y \to X$. Let
$$\bar{f}: N_Y({\Lambda}') \xrightarrow{\cong} N \xrightarrow{f} M \xrightarrow{\cong} 
M_X({\Lambda}).$$
It is enough to show that $\bar{f}$ is $T^n$-homotopic to a $T^n$-homeomorphism. Notice that
$\bar{f}$ also induces the map $\phi$ on the quotients. 
By Proposition \ref{prop-reverse}, $\bar{f} \simeq_{T^n} {\phi}_*$, and, by Corollary \ref{cor-homeomorphism},
${\phi}_*$ is $T$-homotopic to a $T$-homeomorphism $h$ .
Thus 
$$\bar{f} \simeq_{T^n} {\phi}_* \simeq_{T^n} h$$
and the last map is a $T^n$-homeomorphism.
\end{proof}

\begin{rem}
In \cite{yo1}, Theorem 6.2, there is a complete classification of $T^n$-locally standard manifolds. That classification applies to the above result. The difference is that the homeomorphism given in \cite{yo1} it is not necessarily equivariantly homotopic to the original homotopy equivalence.
\end{rem}

The following is an immediate consequence of Theorem \ref{thm-main}. The first is consequence 
of Remark \ref{rem-homeomorphism} (1).

\begin{cor}
Let $M$ be a quasitoric manifold. Let  $N$ a locally linear $T^n$-manifold and $f:N \to M$ a $T^n$-homotopy equivalence. Then $f$ is $T^n$-homotopic to a $T^n$-homeomorphism.
\end{cor}

The following is consequence 
of Remark \ref{rem-homeomorphism} (2).

\begin{cor}
Let $M$ be a smooth $T^n$-locally standard manifold over a manifold with corners 
with all the faces contractible. Let  $N$ a locally linear $T^n$-manifold and $f:N \to M$ a $T^n$-homotopy equivalence. Then $f$ is $T^n$-homotopic to a $T^n$-homeomorphism.
\end{cor}

Also, a slightly more general result holds.

\begin{cor}
Let $M$ be a $T^n$-locally standard  manifold over a manifold with corners $X$. We assume that
$M$ satisfies the conditions of Theorem \ref{thm-main}. Let ${\sigma}: T^n \to T^n$ be a continuous automorphism. Let  $N$ a locally linear $T^n$-manifold and $f:N \to M$ a $\sigma$-equivariant homotopy equivalence. Then $f$ is $\sigma$-homotopic to a $\sigma$-homeomorphism.
\end{cor}

An important class of quasitoric manifolds are complex projective spaces and even dimensional spheres.
Except the case of an even dimensional sphere, the product of these spaces do not have vanishing structure sets. But they are rigid as locally standard torus manifolds.

\end{document}